\documentclass{amsproc}

\newtheorem{theorem}{Theorem}[section]
\newtheorem{lemma}[theorem]{Lemma}
\newtheorem{conjecture}[theorem]{Conjecture}
\newtheorem{corollary}[theorem]{Corollary}

\theoremstyle{definition}
\newtheorem{definition}[theorem]{Definition}

\theoremstyle{remark}

\numberwithin{equation}{section}

\usepackage{amssymb, xcolor}
\usepackage[
backend=biber,
style=alphabetic,
sorting=ynt,
doi=false,
isbn=false,
url=false,
eprint=false
]{biblatex}
\begin{document}

\title{On Some Lie Group Invariant noncommutative functions}


\author{Jiachen Lu}
\address{}
\curraddr{}
\email{}
\thanks{}

\subjclass[2010]{Primary }


\date{}

\dedicatory{}

\begin{abstract}
We study noncommutative polynomials and functions invariant under linear actions of infinite (compact Lie) groups, with a focus on $SU(2)$ acting on two noncommuting variables. For $SU(2)$, we give a complete description of the invariant noncommutative polynomials: there are no invariants of odd degree, and in degree $2n$ the invariants are exactly the span of the $S_{2n}$-orbit of $(XY-YX)^n$. We also construct an explicit basis recursively from the operation $(A,B)\mapsto XAYB-YAXB$, and show that the space of degree-$2n$ invariants has dimension equal to the $n$-th Catalan number. All proofs proceed by direct computation with the Haar integral.
\end{abstract}

\maketitle

\section{Introduction}
\subsection{Noncommutative functions and invariant problem}
        Noncommutative function theory sits at the crossroads of complex analysis, operator theory, and algebra: it studies functions of noncommuting variables, typically evaluated on tuples of matrices or operators. Joseph Taylor initiated the subject through noncommutative spectral theory, developing a difference-differential functional calculus for noncommutative functions \cite{Taylor1,Taylor2,Taylor3}.

        Since then the theory has been developed from several directions and linked to a range of fields. In free probability, Voiculescu introduced the difference quotient operator via his generalized difference quotient (GDQ) ring. In systems and control theory, the same objects arise through Kleene's automata-theoretic framework \cite{system2}. On the purely algebraic side, Rowen gave a detailed treatment of associative noncommutative rings \cite{alg1}, and Cohn studied the skew field of rational noncommutative functions \cite{alg2}. In the theory of ODEs in noncommuting variables, Abduvalieva and Kaliuzhnyi-Verbovetskyi used noncommutative function theory to establish existence and uniqueness of solutions \cite{ncODE}.

        
        The study of symmetric polynomials goes back to Newton's work on the roots of polynomials, and its noncommutative analogue has been investigated in \cite{ncsym1,ncsym2,ncsym3}. In 1916, Noether showed in her paper \cite{Noether1916} that, for a finite group $G$ acting linearly on $V$, $k[V]^G$ is always finitely generated as a ring with at most $|G|$ generators. This classical result inspired a sequence of research in noncommutative setting. As early as 1936, Wolf \cite{Wolf} studied the ring of noncommutative polynomials and showed that, in contrast to Noether's degree bound theorem, it admits no finite basis. This makes a 2013 result of Agler and Young \cite{2vsym} striking: working in two variables, they gave a realization formula for all analytic symmetric noncommutative functions and showed that the skew field of symmetric rational functions admits a finite basis, although the skew field itself is a much larger object than the polynomial ring. The key construction is a pair of eigenfunctions of the swap symmetry $\varphi(X,Y) = (Y,X)$:
        \begin{equation*}
            A(X,Y) = X+Y\qquad B(X,Y)=X-Y
        \end{equation*}
        satisfying $A\circ\varphi = A$ and $B\circ \varphi=-B$. Agler and Young showed that $A,B^2,BAB,BA^2B,\ldots$ generate the algebra of symmetric polynomials, so that $\{A,B^2,BAB\}$ already generates the full skew field of symmetric rational functions. In 2018, Agler, McCarthy, and Young \cite{symman} extended this picture geometrically, constructing a noncommutative manifold and exhibiting a correspondence between holomorphic functions on a two-dimensional manifold and symmetric holomorphic noncommutative functions in two variables.

        This line of work was subsequently generalized in several ways. Cushing, Pascoe, and Tully-Doyle \cite{nvsym} extended the two-variable swap-symmetric picture to a multivariable setting invariant under a finite group equipped with a unitary representation, and Klep, Pascoe, Podlogar, and Volčič \cite{nvsymc} gave the most general result of this type in the setting of realizations, covering multivariable noncommutative functions invariant under arbitrary finite solvable group actions. Most recently, Derksen and Volčič \cite{ncsymext} moved beyond the solvable case entirely: working with skew subfields of a free skew field under the action of an arbitrary finite group $G$, they proved finite generation in general and, for linear actions on a free skew field on $m$ generators, showed the invariant subfield is again free on $|G|(m-1)+1$ generators. They also exhibited a nonlinear $\mathbb{Z}_2$ action for which the invariant subfield fails to be free. This resolves the free Lüroth problem and refutes a 1978 conjecture of Cohn, hence showing that the linearity hypothesis in results of this kind, including our own, is not simply a technical convenience.

\subsection{Infinite group action and our result}
        The results above are confined to finite group actions. In this paper we study two and multiple variable noncommutative functions invariant under infinite (Lie) group actions.
        For $A\in GL_n$ acting on an n-tuple $\mathbf{X^{(n)}}=(X_1,X_2, \ldots, X_n)$ by
        \begin{equation*}
            A\cdot \mathbf{X^{(n)}}=A(X_1,X_2, \ldots, X_n) = (\sum_{i=1}^n A_{1i}X_i, \ldots, \sum_{i=1}^n A_{ni}X_i)
        \end{equation*}
        the noncommutative polynomials invariant under a subgroup $G\subseteq GL_n$ form the algebra
        \begin{equation*}
            \mathbb{C}^G\langle X^{(n)}\rangle = \{f(X^{(n)})\in \mathbb{C}\langle X^{(n)}\rangle\Big| f = f\circ A\; \text{for all } A\in G\}
        \end{equation*}
        Our goal is to identify subgroups $G\subseteq GL_n$ for which $\mathbb{C}^G<X^{(n)}>$ bears a tractable description as a polynomial ring, skew field of rational functions, or holomorphic functions.

        After a general result for invariant free holomorphic functions (Theorem 2.1), we work out the invariant theory for a sequence of group actions on two variables $X,Y$

        We study the case of $G=SO(2)$ in section 3.1. We show that the invariant ring is generated by the eigenfunctions $X+iY$ and $X-iY$, with each invariant subspace in degree $2n$ having dimension$\binom{2n}{n}$.

        The case of $G=SU(2)$ is studied in section 3.4. We prove that the degree $2n$ invariant polynomials are spanned by the $S_{2n}$ orbit of $(XY-YX)^n$. We also construct an explicit basis showing that the dimension of even degree invariant polynomial as a ring is the Catalan number $\frac{1}{n+1}\binom{2n}{n}$, and the dimensions of odd degree ones are 0.

        Finally in section 3.5 we conjecture a generalization from $SU(2)$ to  $SU(n)$, namely that for each fixed degree $mn$, the basis of $SU(n)$ invariant polynomials are the $S_{mn}$ orbit of the m-th power of a degree $n$ basic invariant polynomial.

\subsection{Connection with known results}\label{sec1.3}
        Notice that our $SU(2)$ action on variables $X,Y$ is defined as:
        \begin{align*}
            U\cdot X := U_{11}X+U_{12}Y\\
            U\cdot Y := U_{21}X+U_{22}Y
        \end{align*}
        for $U=\begin{pmatrix}
            U_{11}&U_{12}\\U_{21}&U_{22}
        \end{pmatrix} \in SU(2)$. This can be interpreted as a representation of $SU(2)$ on the vector space $V$ generated by $X$ and $Y$, which is the space of degree 1 polynomials in the noncommuting variables $X$ and $Y$, i.e. $\mathbb{C}_1\langle X,Y\rangle$. Therefore the $SU(2)$ action on degree $n$ polynomials can be understood as the representation of $SU(2)$ on $V^{\otimes n}$, under which degree-$n$ words in $X,Y$ correspond to the standard basis vectors of $V^{\otimes n}$. Our invariant space is then
        \begin{equation*}
            \mathbb{C}_{n}^{SU(2)}\langle X,Y\rangle \;\cong\; \mathrm{Hom}_{SU(2)}\big(\mathbb{C},\,V^{\otimes n}\big),
        \end{equation*}
        the object studied directly by the classical invariant theory of $SU(2)$. Here $\mathrm{Hom}^G_{SU(2)}\big(\mathbb{C},\,V^{\otimes n}\big)$ is the $SU(2)$-invariant maps between two representations, the trivial representation $\mathbb{C}$ and the representation $V^{\otimes n}$. Since $SU(2)$ act on $\mathbb{C}$ trivially, for any $\phi \in \mathrm{Hom}_{SU(2)}\big(\mathbb{C},\,V^{\otimes n}\big)$,
        \begin{equation*}
            \phi(1) = \phi(U\cdot 1) = U\cdot \phi(1)
        \end{equation*}
        therefore the image of $\phi$ is the $SU(2)$ invariant tensors, hence the $SU(2)$ invariant polynomials we are interested in. We use the $\mathrm{Hom}$ notation below only because it is what the cited literature uses.

        We record this connection both to situate our results relative to the literature and to explain why we nonetheless pursue an elementary, representation-theory-free route in Section~\ref{sec:algebra} below.

        \subsubsection{The classical result}\label{subsubsec:classical}
        That $SU(2)$-invariant tensors in $V^{\otimes 2n}$ are counted by the Catalan number $C_n$, with a basis indexed by non-crossing pairings of $2n$ points, is classical. It originates in the \emph{binor calculus} of Rumer, Teller, and Weyl \cite{RTW32}, later used by Penrose to define \emph{spin networks} \cite{Pen71} (trivalent graphs evaluated using the antisymmetrizer $\varepsilon_{AB}$), and independently underlies the Temperley--Lieb algebra $TL_n(\delta)$ \cite{TL71}, introduced via transfer matrices in statistical mechanics. The two pictures agree. On the basis $\{X,Y\}$ of $V=\mathbb{C}^2$, define
        \begin{equation*}
            \omega(v,w) := v_1w_2 - v_2w_1, \qquad v=v_1X+v_2Y,\ w=w_1X+w_2Y,
        \end{equation*}
        i.e.\ $\omega(X,Y)=1=-\omega(Y,X)$ and $\omega(X,X)=\omega(Y,Y)=0$: this is preserved by $SU(2)$. Writing $c=X\otimes Y-Y\otimes X$, define $e:V\otimes V\to V\otimes V$ and $e_i$ by
        \begin{equation*}
            e(v\otimes w) := \omega(v,w)\,c, \qquad e_i := \mathrm{id}^{\otimes(i-1)}\otimes e\otimes \mathrm{id}^{\otimes(n-i-1)}.
        \end{equation*}
        Recall that the Temperley--Lieb algebra $TL(\delta)$ is the unital associative algebra generated by $e_i's$ with the relations
        \begin{align*}
            e_i^2 &= \delta e_i\\
            e_i e_{i+1} e_i &= e_{i}\\
            e_i e_{i-1} e_i &= e_{i}\\
            e_i e_j &= e_j e_i \; \text{for }|i-j|\geq2
        \end{align*}
        The $e_i$ satisfy the Temperley--Lieb relations at $\delta=2$, giving an algebra map $TL_n(2)\to\mathrm{End}(V^{\otimes n})$ whose image is the \emph{full} centralizer $\mathrm{End}_{SU(2)}(V^{\otimes n})$. A modern, unified treatment is given in Kauffman and Lins \cite{KL94} through a diagrammatic aprroach.

        \subsubsection{Identification with our results}\label{subsubsec:identification}
        The correspondence with our own notation is term-by-term:
        \begin{center}
        \begin{tabular}{|p{5.3cm}|p{4.0cm}|c|}
            \hline
            \textbf{Our notation} & \textbf{Classical correspondent} & \textbf{Ref.} \\
            \hline
            $XY-YX$ & coevaluation $c$ of $\omega$ (Penrose's $\varepsilon_{AB}$) & --- \\
            \hline
            $P_n=\{XAYB-YAXB\}$ & recursive non-crossing basis & Thm.~\ref{part2} \\
            \hline
            $\dim \mathbb{C}^{SU(2)}_{2n}\langle X,Y\rangle = \frac{1}{n+1}\binom{2n}{n}$ & Catalan count & Thm.~\ref{conjSU2} \\
            \hline
        \end{tabular}
        \end{center}
        We do not reprove this correspondence in diagrammatic language here, but it can be checked directly, and no part of Theorems~\ref{simpSU2}--\ref{part2} below is in tension with it.

        \subsubsection{What is new here}\label{subsubsec:new}
        We do not claim Theorems~\ref{simpSU2}--\ref{part2} below as a new classification. The contribution we believe is new, is the \emph{method of proof}: our argument in Section~\ref{sec:algebra} proceeds entirely by direct manipulation of the Haar integral on $SU(2)$ (Definition~\ref{def:avgI}, Lemmas~\ref{etaI}--\ref{I(M-N)}, and the induction in Theorems~\ref{part1}, \ref{part2}), invoking no Schur--Weyl duality, Temperley--Lieb algebra, or diagrammatic calculus. It is an independent, elementary verification of the classical result from first principles.

        The significance of the new proof is that there is hope that it may generalize to situation that the existing results do not cover. For example there are examples where the dimension of degree 3k $SU(3)$ invariant nc polynomials doesn't match the dimension of generalized T-L algebra. 
        
        \subsubsection{The general $SU(n)$ case: Weyl's First Fundamental Theorem}\label{subsubsec:weylFFT}
        The phenomenon underlying $D_2:=XY-YX$ is not special to $n=2$. Weyl's \emph{First Fundamental Theorem of Classical Invariant Theory} \cite{Wey39} shows that for $G=SL(n)$ acting on tensor powers of its standard $n$-dimensional representation $V$, every invariant of arbitrarily many copies of $V$ is a polarization of an invariant of exactly $n$ copies -- equivalently, that the top exterior power $\Lambda^nV\cong\mathbb{C}$ is the atomic invariant, first appearing at degree $n=\dim V$. Both $D_2$ and the degree-$n$ invariant underlying our conjecture in Section~\ref{sec:SUn} are instances of this single construction:
        \begin{equation*}
            D_n := \sum_{\eta\in S_n}\mathrm{sgn}(\eta)\,\eta\curvearrowright(X_1X_2\cdots X_n) \;=\; X_1\wedge\cdots\wedge X_n \;\in\; V^{\otimes n}.
        \end{equation*}
        Here the permutation action is defined as
        \begin{equation*}
            \eta\curvearrowright(X_1X_2\cdots X_n):= (X_{\eta(1)}X_{\eta(2)}\cdots X_{\eta(n)})
        \end{equation*}

        \subsubsection{Diagrammatic bases at every rank}\label{subsubsec:diagbases}
        The combinatorial refinement of Weyl's theorem --- an explicit basis, not merely a generating set --- is likewise known at every rank, not only $n=2$, and in every case the construction is fundamentally \emph{diagrammatic}. Brauer \cite{Bra37} built a diagram algebra of cup--cap and through-strand tangles for $O(n)$ and $Sp(n)$, specializing at $SU(2)\cong Sp(2)$ to the Temperley--Lieb picture of \S\ref{subsubsec:classical}. Since $SU(3)$'s standard representation is not self-dual, Kuperberg \cite{Kup96} instead used oriented trivalent graphs (\emph{spiders}), with a source/sink vertex standing in for $TL$'s cup; Cautis, Kamnitzer, and Morrison \cite{CKM14} extended this to labeled trivalent \emph{webs}, presenting $\mathrm{Hom}_{SU(n)}(\mathbb{C},V^{\otimes kn})$ for all $SU(n)$ via quantum skew Howe duality. In each case the basis is exhibited graphically --- generators and relations for a category of diagrams --- unlike the direct Haar-integral manipulation used throughout Section~\ref{sec:algebra}.

        \subsubsection{Why the elementary route still matters}\label{subsubsec:whyelementary}
        Given \S\S\ref{subsubsec:weylFFT}--\ref{subsubsec:diagbases}, one might ask why the Haar-integral approach of Section~\ref{sec:algebra} is worth pursuing at all, rather than simply invoking Weyl's theorem or specializing the \cite{CKM14} web basis to our setting. Obtaining an explicit combinatorial description of invariants at general rank has historically required substantial new machinery at each step: a new trivalent vertex and oriented strands for $SU(3)$ alone in \cite{Kup96}, and a full quantum-skew-Howe-duality argument to reach general $n$ in \cite{CKM14} --- a question that remained open for nearly two decades between the two papers. Moreover, translating the resulting web bases into the specific noncommutative-polynomial language of this paper --- an explicit spanning set and dimension formula in the style of Theorem~\ref{conjSU2}, stated directly for $\mathbb{C}^{SU(n)}\langle X_1,\ldots,X_n\rangle$ --- is not immediate from \cite{CKM14} as stated, and would itself require nontrivial translation work. The Haar-integral method of Section~\ref{sec:algebra}, by contrast, is stated for a general compact Lie group from the outset (Definition~\ref{def:avgI} makes no essential use of $SU(2)$), and Theorems~\ref{simpSU2}--\ref{part2} show it is already powerful enough to independently recover the correct classical answer when $G=SU(2)$, directly in our own notation. This is the basis for the conjecture of Section~\ref{sec:SUn}: that the same technique, applied directly, may reach an explicit answer for $SU(n)$, $n>2$, more directly in our setting than reformulating the existing web-basis literature would.

    \section{Analysis Part}
        Before the algebraic classification that occupies the rest of the paper, we record a bridge to the analytic setting: for an arbitrary Lie group $G$, every $G$-invariant free holomorphic function nonvanishing at the origin already lies in the free closure of $G$-invariant nc polynomials, so the classification results of Section~\ref{sec:algebra} extend automatically to the analytic setting. The proof is short: the nc Oka--Weil theorem \cite{ncoka} expands $f$ as a series of homogeneous nc polynomials, and since $G$ preserves degree, invariance of $f$ forces each homogeneous piece to be invariant on its own.
        \begin{theorem}
            For any Lie group action $G$, all G-invariant free holomorphic function that is non vanishing at 0 is contained in the free closure of G-invariant nc polynomials
        \end{theorem}

        \begin{proof}
            Let $f$ be a free holomoprphic function such that $f(\mathbf{B}) = f(A\mathbf{B})$ for any $A\in G$ and $f(0) \neq 0$. By nc Oka-Weil theorem \cite{ncoka}, $f$ can be written as a series of homogeneous nc polynomails, i.e.
            \begin{equation*}
                f = \sum_{i=0}^{\infty} f_n
            \end{equation*}
            where $f_n$ is an homogeneous nc-polynomial of degree $n$. Then since $f$ is invariant under $G$ actions, each of $f_n$ also have to be invariant under $G$ actions.
        \end{proof}

    \section{Algebra Part}\label{sec:algebra}

	\subsection{The torus $SD (n)$}
		By $SD(n)$ we mean the set of diagonal matrices in $SL(n)$
		\begin{equation}
			D = diag(e^{i\theta_1},e^{i\theta_2},\ldots,e^{i\theta_{n}})
		\end{equation}
		  with $\sum_{1}^{n}\theta_j = 0$. So an $SD(n)$ action on $\mathbf{X^{(n)}}$ is
		  \begin{equation}
		 	D \cdot (X_1,\ldots,X_{n}) = (e^{i\theta_1}X_1,\ldots,e^{i\theta_{n}}X_{n})
		  \end{equation}
        Then there is a natural way to make it invariant by averaging it out. We define $P: \mathbb{C}\langle \mathbf{X}^{(n)}\rangle \rightarrow \mathbb{C}^{SD(2)}\langle \mathbf{X}^{(n)} \rangle$ by
        \begin{equation*}
            P\big(F(\mathbf{X}^{(n)})\big) := \int_{SD(n)} D\cdot F(\mathbf{X}^{(n)}) dm(D)
        \end{equation*}
        where $m(D)$ is the bi-invariant normalized Haar measure on $SD(n)$.
        Its image is invariant under any $SD(2)$ action $A$ since
        \begin{align*}
            &A\cdot P\big(F(\mathbf{X}^{(n)})\big) \\
            =&\int_{SD(n)} A\cdot\big(D\cdot F(\mathbf{X}^{(n)})\big) dD\\
            =&\int_{SD(n)} (DA)\cdot F(\mathbf{X}^{(n)}) dD\\
            =&P\big(F(\mathbf{X}^{(n)})\big)
        \end{align*}
        And $P$ is a projection because
        \begin{align*}
            P^2F &= \int_{SD(2)} G\cdot (PF)\;dG\\
            &=\int_{SD(2)} PF\;dG\\
            &=PF
        \end{align*}

        \begin{theorem}\label{balancedegree}
            The polynomials that are invariant under the $SD(n)$ action form a vector space generated by words in $X_1,\ldots,X_{n}$ such that each variable has the same degree
	 	\begin{equation*}
	 		\mathbb{C}^{SD(n)}<X_1,\ldots,X_{n+1}> =span\{\prod_k \prod_{j=1}^{n} X_j^{w_j^k} | \sum_k w_j^k \text{are the same for all j}\}
	 	\end{equation*}

        And the degree $kn$ homogeneous $SD(2)$ invariant polynomials have degree $kn$ choose $n,n,\ldots,n$
        \begin{equation*}
            deg(\mathbb{C}_{kn}^{SD(n)}) = \binom{kn}{n,n,\ldots,n}
        \end{equation*}
        \end{theorem}

        \begin{proof}
            Explicitly, the Haar measure $m$ on $SD(n)$ is given by
            \begin{equation*}
                m\Big(diag(e^{i\theta_1},\ldots,e^{i\theta_{n-1}},e^{-i\sum_{j=1}^{n-1}\theta_j})\Big) = \frac{1}{(2\pi)^{n-1}}d\theta_1\ldots d\theta_{n-1}
            \end{equation*}
            Then the image of the word $w=\prod_k \prod_{j=1}^{n} X_j^{w_j^k}$ under $P$ is
            \begin{align*}
                Pw &= \int_{SD(2)} D\cdot w\;dm(D)\\
                &= \int_{SD(2)} D\cdot \prod_k \prod_{j=1}^{n} X_j^{w_j^i}\;dm(D)\\
                &= \frac{1}{(2\pi)^{n-1}} \int_{0}^{2\pi}\ldots\int_{0}^{2\pi} \prod_{k} \Big[\Big(\prod_{j=1}^{n-1} (e^{i\theta_j}X_j)^{w_j^k}\Big)(e^{-i\sum_{j=1}^{n-1}\theta_j}X)^{w_n^k}\Big] d\theta_1\ldots d\theta_{n-1}\\
                &= \frac{1}{(2\pi)^{n-1}} \int_{0}^{2\pi}\ldots\int_{0}^{2\pi} \Big[\prod_{j=1}^{n-1}e^{i\sum_k (w_j^k-w_n^k)\theta_j}\Big] \prod_k \prod_{j=1}^{n} X_j^{w_j^i}\;d\theta_1\ldots d\theta_{n-1}\\
                &=\Big[\frac{1}{(2\pi)^{n-1}} \int_{0}^{2\pi}\ldots\int_{0}^{2\pi} \Big(\prod_{j=1}^{n-1}e^{i\sum_k (w_j^k-w_n^k)\theta_j}\Big) \;d\theta_1\ldots d\theta_{n-1} \Big] w
            \end{align*}
            When $\sum_k (w_j^k-w_n^k) =0$ for all $j$, $Pw=w$. This proves the first part of the theorem.

            Because each length $kn$ word is invariant when the degrees of all letters are the same, there are $\binom{kn}{n,n,\ldots,n}$ many such words in degree $kn$.
        \end{proof}

    \subsection{Structure of $\mathbb{C}^{SO(2)} \left\langle X,Y \right\rangle $}
        In this section, we show that the geometric equivalence of $SO(2)$ and $SD(2)$ gives the equivalence of polynomials that are invariant under $SO(2)$ and those that are invariant under $SD(2)$.
        \begin{theorem}
			$\mathbb{C}^{SO(2)} \left\langle X,Y \right\rangle $ is isomorphic to $\mathbb{C}^{SD(2)} \left\langle X,Y \right\rangle $.
		\end{theorem}
        \begin{proof}
            Consider $A=\frac{1}{\sqrt{2}}\begin{pmatrix}
                1&i\\i&1
            \end{pmatrix}\in SO(2)$.
            Let $A$ be a single matrix action. Then $A$ maps $SD(2)$ invariant matrices to $SO(2)$ invariant matrices. Indeed, for any $SD(2)$ invariant matrix $f(X,Y)$ and an $SO(2)$ action $B$,
            \begin{align*}
                &B\cdot\big(A\cdot f(X,Y)\big)\\
                =&(BA)\cdot f(X,Y)\\
                =&f(X,Y)
            \end{align*}
            This is because $BA\in SO(2)$.
            Since $A$ is invertible as a matrix, the ring homomorphism defined by $A$ is an isomorphism.
        \end{proof}
        In fact, since the action $A$ preserves the length of each word, on each degree $k$, $\mathbb{C}^{SO(2)}_k \left\langle X,Y \right\rangle $ is isomorphic to $\mathbb{C}^{SD(2)}_k \left\langle X,Y \right\rangle $.

        \begin{corollary}
            The $SO(2)$ invariant nc polynomials of degree $k$ $\mathbb{C}^{SO(2)}_k$ is a vector space generated by words in same number of letters $X+iY$ and $Y+iX$, and the degree is $\binom{2n}{n}$.
        \end{corollary}

    \subsection{The relationship between $\mathbb{C}^{SO(n)}$ and $\mathbb{C}^{SU(n)}$}
        Here we show that nc polynomials that are invariant under $SU(2)$ actions are the same as those that are invariant under both $SD(2)$ and $SO(2)$ actions.
        \begin{lemma} \label{UOD}
        The $SU(2)$ invariant polynomials satisfy
        \begin{equation}
            \mathbb{C}^{SU(n)} = \mathbb{C}^{SO(n)}\cap \mathbb{C}^{SD(n)}
        \end{equation}
        \end{lemma}
        One direction of the inclusion $\mathbb{C}^{SU(n)} \subseteq \mathbb{C}^{SO(n)}\cap \mathbb{C}^{SD(n)}$ is trivial and the other direction is a result of the following lemma
        \begin{lemma}
            For any $U \in SU(n)$, there exists $A,B\in SO(n)$ and $D \in SD(n)$ such that $U=ADB$
        \end{lemma}
        \begin{proof}
            Since $A$ and $B$ are two real orthogonal matrices, we only need to show the real and imaginary part of $U$ have a singular value decomposition by $A$ and $B$ simultaneously.

            Write $U = U_r + iU_i$, Since $U$ is special unitary,
            \begin{equation*}
                UU^* = U_rU_r^T + U_iU_i^T + i(-U_rU_i^T + U_iU_r^T)=I
            \end{equation*}
            Let $A,B$ be the matrices that provides a singular decomposition of $U_r$.
            \begin{equation*}
                U_r = AD_rB
            \end{equation*}
            Then by the definition of singular decomposition, $A^T U_rU_r^T A $ is diagonal. Because
            \begin{equation*}
                A^T U_iU_i^T A = A^T (I-U_rU_r^T) A
            \end{equation*}
            $S$ diagonalize $U_iU_i^T$.
            Similarly $B$ diagonalize $U_i^TU_i$. By \cite{SVD} chapter 7, $A,B$ also give a singular decomposition $U_i = AD_iB$.

            Lastly since $U,A,B$ are all special unitary matrices, $D = D_r+iD_i$ is also special unitary, so $D \in SD(n)$.
        \end{proof}

	\subsection{$SU(2)$}
		\begin{theorem}\label{simpSU2}
		    The $SU(2)$ invariant nc-polynomials $\mathbb{C}^{SU(2)}\langle X,Y\rangle$ are the $SO(2)$ invariant polynomials that in addition satisfy $f(X,Y) = f(\frac{X+iY}{\sqrt{2}}, \frac{Y+iX}{\sqrt{2}})$, i.e.
            \begin{equation*}
                \mathbb{C}^{SU(2)}\langle X,Y \rangle = \{f | f(X,Y) = f(\frac{X+iY}{\sqrt{2}}, \frac{Y+iX}{\sqrt{2}}), f\in\mathbb{C}^{SO(2)}\langle X,Y\rangle \}
            \end{equation*}
		\end{theorem}

        \begin{proof}
            We have $SO(2)\subset SU(2)$ and $\begin{pmatrix}
                \frac{1}{\sqrt{2}} &i\frac{1}{\sqrt{2}}\\i\frac{1}{\sqrt{2}}&\frac{1}{\sqrt{2}}
            \end{pmatrix} \in SU(2)$. In other words,
            \begin{equation*}
                \mathbb{C}^{SU(2)}\langle X,Y \rangle \subseteq \{f | f(X,Y) = f(\frac{X+iY}{\sqrt{2}}, \frac{Y+iX}{\sqrt{2}}), f\in\mathbb{C}^{SO(2)}\langle X,Y\rangle \}
            \end{equation*}

            For the other direction, let
            \begin{equation*}
                f \in \{f | f(X,Y) = f(\frac{X+iY}{\sqrt{2}}, \frac{Y+iX}{\sqrt{2}}), f\in\mathbb{C}^{SO(2)}\langle X,Y\rangle \}
            \end{equation*}
            Since $f \in \mathbb{C}^{SO(2)}\langle X,Y\rangle$, $f$ can be written as a linear combination of words in letters $(x+iY)$ and $(X-iY)$ with each word having same number of the two letters. That is,
            \begin{equation*}
                f = \sum_{(w,v)\in W} (a_{w,v}\prod_{j}(X+iY)^{w_j}(X-iY)^{v_j})
            \end{equation*}
            where $W$ is the set of finite sequences pairs $(w,v)$ such that $\sum_j w_j = \sum_j v_j$. Since $f(X,Y) = f(\frac{X+iY}{\sqrt{2}}, \frac{Y+iX}{\sqrt{2}})$, \begin{align*}
                f &= \sum_{(w,v)\in W} (a_{w,v}\prod_{j}(X+iY)^{w_j}(X-iY)^{v_j})\\
                &= \sum_{(w,v)\in W} (a_{w,v}\prod_{j}(X+iY)^{w_j}(\frac{Y+iX}{i})^{v_j})\\
                &= \sum_{(w,v)\in W} (a_{w,v}\prod_{j}(\sqrt{2}X)^{w_j}(\frac{\sqrt{2}Y}{i})^{v_j})\\
                &= \sum_{(w,v)\in W} (2^{\sum w_j}a_{w,v}\prod_{j} X^{w_j} Y^{v_j})
            \end{align*}
            So $\sum_j w_j = \sum_j v_j$ gives $f \in \mathbb{C}^{SD(2)}\langle X,Y\rangle$. By (\ref{UOD}), $f \in \mathbb{C}^{SU(2)}\langle X,Y \rangle$.
        \end{proof}

        To give a context of where the next theorem comes from, consider the mapping $T: f(X,Y) \rightarrow f(\frac{X+iY}{\sqrt{2}}, \frac{Y+iX}{\sqrt{2}})$. $T$ is a linear map on $\mathbb{C}\langle X,Y\rangle$. By (\ref{UOD},\ref{simpSU2}), $\mathbb{C}^{SU(2}$ is the eigenspace of $PTP$ corresponding to eigenvalue 1, where $P$ is the projection from $\mathbb{C}\langle X,Y\rangle$ to $\mathbb{C}^{SD(2)} \langle X,Y\rangle$ that keeps the balanced degree words in a polynomial.

        Restrict the operator $PTP$ to $\mathbb{C}_{2n}\langle X,Y\rangle$, and denote it as $(PTP)_n$, then $(PTP)_n$ is a matrix of size $\binom{2n}{n}\times\binom{2n}{n}$. Pick any ordering of length $2n$ words in $X$ and $Y$ with n $X$'s and n $Y$'s as basis of $\mathbb{C}_{2n}\langle X,Y\rangle$, denote them as $W_1^{(2n)},\ldots,W_{\binom{2n}{n}}^{(2n)}$, then any polynomial in $\mathbb{C}_{2n}\langle X,Y\rangle$ can be identified as a length $\binom{2n}{n}$ column vector $v$ and the i,j-th entry of $(PTP)_n$ is
        \begin{equation}\label{ij}
            \Delta(W_i^{(2n)},W_j^{(2n)}) := \frac{(-1)^k}{2^{n}}
        \end{equation}
        where $k$ is the number of places at which $W_i^{(2n)}$ is different from $W_j^{(2n)}$.

        Note that since $\begin{pmatrix}
            0&1\\-1&0
        \end{pmatrix} \in SU(2)$, $(PTP)_n$ is always the top left block of itself when cut into 4 equally sized blocks tenser with $\begin{pmatrix}
            1&(-1)^n\\(-1)^n&1
        \end{pmatrix}$, which is useful in calculation.

        From the definition of $(PTP)_n$ (\ref{ij}), the matrix $(PTP)_n$ is independent of the order that each word is read in. Indeed, for any $\eta \in S_{2n}$ being a permutation of letters in each word. For a word $\overline{A_1,\ldots,A_{2n}}$, let $\eta \curvearrowright\overline{A_1,\ldots,A_{2n}} = \overline{A_{\eta(1)},\ldots,A_{\eta(2n)}}$. Then for a new ordering of length $2n$ words $V_j^{(2n)} = \eta\curvearrowright W_j^{(2n)}$, $\Delta(W_i^{(2n)},W_j^{(2n)})=\Delta(V_i^{(2n)},V_j^{(2n)})$. By an abuse of notation let $\eta$ be a $\binom{2n}{n}\times\binom{2n}{n}$ matrix representing the base change, we have $\eta (PTP)_n\eta^{-1} = (PTP)_n$ for all $\eta \in S_{2n}$.

        Now define the permutation $\eta$ on a degree $2n$ homogeneous polynomial $p=\sum a_jp_j$ by $\eta\curvearrowright p=\sum_j a_j (\eta\curvearrowright p_j)$ for monomials $p_j$. If in addition $p_j$ is the choosen basis $W_j^{(2n)}$, then $p$ corresponds to the column vector $(a_1,\ldots,a_{\binom{2n}{n}})^T$ and $\eta\curvearrowright p$ corresponds to $\eta\times(a_1,\ldots,a_{\binom{2n}{n}})^T$.

        Since $\eta (PTP)_n\eta^{-1} = (PTP)_n$ and $(XY-YX)^n$ is an obvious invariant polynomial in $\mathbb{C}_{2n}^{SU(2)}$, $\eta \curvearrowright (XY-YX)^n$ is also an invariant polynomial. It is natural to conjecture that this gives a full basis, and this turns out to be true.
        \begin{theorem} \label{conjSU2}
            \begin{equation*}
                \mathbb{C}_{2n}^{SU(2)} = span\{ \eta \curvearrowright(XY-YX)^n | \eta\in S_{2n}\}
            \end{equation*}
            and
            \begin{equation*}
                dim(\mathbb{C}_{2n}^{SU(2)}\langle X,Y\rangle = \frac{1}{n+1}\binom{2n}{n}
            \end{equation*}
        \end{theorem}
        This theorem will be proved in the following sections.

        \subsection{Proof of Theorem \ref{conjSU2}}

        We firstly show that $\mathbb{C}_{2n}^{SU(2)} = span\{ \eta \curvearrowright(XY-YX)^n | \eta\in S_{2n}\}$ by induction.

        \begin{definition}\label{def:avgI}
            Let $dh$ be the normalized Haar measure on $SU(2)$. We define the averaging operator that yields an invariant function as
            \begin{equation*}
                I:f(X,Y) \mapsto \int_{SU(2)}u\cdot f (X,Y)d h(u)
            \end{equation*}
        \end{definition}

        Since $h$ is the Haar measure, fix any $v\in SU(2)$,
        \begin{align*}
            v(I(f)) =& v\Big(\int_{SU(2)}u\cdot f (X,Y)d h(u)\Big)\\
            = & \int_{SU(2)}v\cdot (u\cdot f(X,Y))d h(u)\\
            =& \int_{SU(2)}(uv)\cdot f (X,Y)d h(u)\\
            =& \int_{SU(2)}(uv)\cdot f (X,Y)d h(uv)\\
            =& I(f)
        \end{align*}
        So $I$ maps any non commutative polynomial to an $SU(2)$ invariant polynomial.
        \newline
        \begin{lemma}\label{etaI}
            If $f$ is a degree $n$ non-commutative polynomial, and $\eta\in S_n$ is a permutation, then $I(\eta \curvearrowright f) = \eta \curvearrowright I(f)$
        \end{lemma}

        \begin{proof}
        By linearity, we only need to show the case when $f$ is a word. Assume $f(X,Y)$ is a word $\overline{A_1A_2\ldots A_n}$, with $A_i$ being $X$ or $Y$. Then $\eta \curvearrowright f = \overline{A_{\eta^{-1}(1)}\ldots A_{\eta^{-1}(n)}}$.

        Assume $I(f) = \sum_{j=1}^{2^n}c_j \overline{B_{j1}\ldots B_{jn}}$, where $\overline{B_{i1}\ldots B_{in}}$ are also words in $X,Y$. As $j$ ranges from 1 to $2^n$, $\{\overline{B_{i1}\ldots B_{in}}\}_{j}$ covers all words of length $n$.

        Explicitly, for $u = \begin{pmatrix}
            u_{XX}&u_{XY}\\u_{YX}&u_{YY}
        \end{pmatrix}$,
        \begin{align*}
            I(f) &= \int_{SU(2)}f\circ u (X,Y)d h(u)\\
            &= \int_{SU(2)}f (u_{\scriptscriptstyle XY}X+u_{\scriptscriptstyle XY}Y,u_{\scriptscriptstyle YX}X+u_{\scriptscriptstyle YY}Y)d h(u)\\
            &= \int_{SU(2)} \sum_{j=1}^{2^n}\prod_{i=1}^{n}u_{\scriptscriptstyle A_i B_{ji}}\overline{B_{j1}\ldots B_{jn}}dh(u)\\
            &= \sum_{j=1}^{2^n}\Big(\int_{SU(2)} \prod_{i=1}^{n}u_{\scriptscriptstyle A_i B_{ji}}dh(u)\Big)\overline{B_{j1}\ldots B_{jn}}\;.
        \end{align*}

        Now compute $I(\eta \curvearrowright f)$
        \begin{align*}
            I(\eta \curvearrowright f) &= \int_{SU(2)}(\eta \curvearrowright f)\circ u (X,Y)d h(u)\\
            &= \int_{SU(2)} \sum_{j=1}^{2^n}\prod_{i=1}^{n}u_{\scriptscriptstyle A_{\eta^{-1}(i)} B_{ji}}\overline{B_{j1}\ldots B_{jn}}dh(u)\\
            &= \sum_{j=1}^{2^n}\Big(\int_{SU(2)} \prod_{i=1}^{n}u_{\scriptscriptstyle A_{\eta^{-1}(i)} B_{ji}}dh(u)\Big)\overline{B_{j1}\ldots B_{jn}}\\
            &= \sum_{j=1}^{2^n}\Big(\int_{SU(2)} \prod_{i=1}^{n}u_{\scriptscriptstyle A_{i} B_{j\eta(i)}}dh(u)\Big)\overline{B_{j1}\ldots B_{jn}}\;.
        \end{align*}
        Since $\overline{B_{j1}\ldots B_{jn}}$ covers all possible word of length n, for the word $\overline{B_{j\eta(1)}\ldots B_{j\eta(n)}}$, there exist another $j'$ such that $\overline{B_{j\eta(1)}\ldots B_{j\eta(n)}} = \overline{B_{j'1}\ldots B_{j'n}}$.

        So
        \begin{equation*}
            \int_{SU(2)} \prod_{i=1}^{n}u_{\scriptscriptstyle A_{i} B_{j\eta(i)}}dh(u) = \int_{SU(2)} \prod_{i=1}^{n}u_{\scriptscriptstyle A_{i} B_{j'i}}dh(u)
        \end{equation*}
        hence
        \begin{align*}
            I(\eta \curvearrowright f) &= \sum_{j'=1}^{2^n}\Big(\int_{SU(2)} \prod_{i=1}^{n}u_{\scriptscriptstyle A_{i} B_{j'i}}dh(u)\Big)\overline{B_{j'\eta^{-1}(1)}\ldots B_{j'\eta^{-1}(n)}}\\
            &= \eta \curvearrowright I(f)\;.
        \end{align*}
        \end{proof}

        \begin{lemma}\label{I(M-N)}
            Let $M(X,Y)$ and $N(X,Y)$ be two words in $\mathbb{C}_{2k+2}^{SD(2)}$, i.e. $M$ and $N$ are of degree $2k+2$, in which the degree of $X$ and $Y$ are both k+1. Assuming
            \begin{equation*}
                \mathbb{C}_{2k}^{SU(2)} = span\{ \eta \curvearrowright(XY-YX)^k | \eta\in S_{2k}\}
            \end{equation*}
            then
            \begin{equation*}
                I(M-N) \in span\{ \eta \curvearrowright(XY-YX)^{2k+2} | \eta\in S_{2k+2}\}
            \end{equation*}
        \end{lemma}

        \begin{proof}
            Since $M(X,Y)$ is a product of $X$'s and $Y$'s, let $M = PYXQ$ and suppose $N=PXYQ$ for two words $P,Q$ in $X,Y$ then $M-N=P(YX-XY)Q$.

            Suppose the degree of $P$ is $i-1$. Let $\eta = (1, i)(2, i+1)\in S_{2k+2}$, then
            \begin{equation*}
                \eta\curvearrowright\Big(P(X,Y)(YX-XY)Q(x,Y)\Big)=(YX-XY)P(X,Y)Q(x,Y)
            \end{equation*}

            By Lemma \ref{etaI},
            \begin{align*}
                I(M-N) &= \eta^{-1}\curvearrowright I(\eta \curvearrowright (M-N))\\
                &= \eta^{-1}\curvearrowright \int_{SU(2)}\eta\curvearrowright(M-N)\circ u \;dh(u)\\
                &=\eta^{-1}\curvearrowright\int_{SU(2)}\eta\curvearrowright\Big(P(X,Y)(YX-XY)Q(x,Y)\Big)\circ u \;dh(u)\\
                &=\eta^{-1}\curvearrowright\int_{SU(2)}\Big((YX-XY)P(X,Y)Q(x,Y)\Big)\circ u \;dh(u)\\
                &=\eta^{-1}\curvearrowright\int_{SU(2)}\Big((YX-XY)\circ u\Big)\Big(P(X,Y)Q(x,Y)\Big)\circ u \;dh(u)\\
                &=\eta^{-1}\curvearrowright\int_{SU(2)}(YX-XY)\Big(P(X,Y)Q(x,Y)\Big)\circ u \;dh(u)\\
                &=\eta^{-1}\curvearrowright\Big((YX-XY)\int_{SU(2)}\Big(P(X,Y)Q(x,Y)\Big)\circ u \;dh(u)\Big)\\
                &=\eta^{-1}\curvearrowright\Big((YX-XY)I(PQ)\Big)
            \end{align*}
            Since $PQ$ is a degree $2k$ polynomial, $I(PQ)$ is an $SU(2)$ invariant non commutative. By the assumption,
            \begin{equation*}
                I(PQ)\in \mathbb{C}_{2k}^{SU(2)} = span\{\rho\curvearrowright(XY-YX)^{2k}\;|\;\rho\in S_{2k}\}
            \end{equation*}
            Then
            \begin{equation*}
                (YX-XY)I(PQ)\in span\{\rho\curvearrowright(XY-YX)^{2k+2}\;|\;\rho\in S_{2k+2}\}
            \end{equation*}
            Thus
            \begin{equation*}
                I(M-N)=\eta^{-1}\curvearrowright(YX-XY)I(PQ)\in span\{\rho\curvearrowright(XY-YX)^{2k+2}\;|\;\rho\in S_{2k+2}\}
            \end{equation*}
            For arbitrary $M$ in $\mathbb{C}^{SD(2)}_{2k+2}$, there is a finite sequence $M=A_0, A_1,\ldots,A_k=N$ such that each neighboring pair $(A_j,A_{j+1})$ satisfy that $A_j = PYXQ$ and $A_{j+1}=PXYQ$
        \end{proof}

        \begin{lemma}\label{Ikernel}
            If $G$ is a monomial in $\mathbb{C}<X,Y>\ominus\;\mathbb{C}^{SD(2)}<X,Y>$,
            \begin{equation*}
                \int_{SU(2)}G\circ u\;dh(u)=0
            \end{equation*}
        \end{lemma}
        \begin{proof}\label{non homo int}
            Let $v = \begin{pmatrix}
                e^{i\theta}&0\\0&e^{-i\theta}
            \end{pmatrix} \in SU(2)$. Then
            \begin{align*}
                &\int_{SU(2)}G\circ u\;dh(u)\\
                =&\int_{SU(2)}G\circ (vu)\;dh(u)\\
                =&\int_{SU(2)}G\circ v\circ u\;dh(u)\\
            \end{align*}

            Given that $\mathbb{C}^{SD(2)}<X,Y>$ contains all words in $X,Y$ such that the degree of $X$ equals the degree of $Y$, the degree of $X$ and the degree of $Y$ are different in $G$.

            Consider $G\circ v$. $v$ maps $(X,Y)$ to $(e^{i\theta}X,e^{-i\theta}Y)$, so $G\circ v = e^{ik\theta}G$, where $k$ is the difference between the degree of $X$ and the degree of $Y$. Since $G \in \mathbb{C}<X,Y>\ominus\mathbb{C}^{SD(2)}<X,Y>$, $k\neq 0$. So,
            \begin{align*}
                &\int_{SU(2)}G\circ u\;dh(u)\\
                =&\int_{SU(2)}G\circ v\circ u\;dh(u)\\
                =&\int_{SU(2)}(e^{ik\theta}G)\circ u\;dh(u)\\
                =&e^{ik\theta}\int_{SU(2)}G\circ u\;dh(u)\\
            \end{align*}
            This hold for all $\theta$, so $\int_{SU(2)}G\circ u\;dh(u)=0$
        \end{proof}

        \begin{theorem}\label{part1}
            $\mathbb{C}_{2n}^{SU(2)} = span\{ \eta \curvearrowright(XY-YX)^n | \eta\in S_{2n}\}$
        \end{theorem}
        \begin{proof}
            When $n=1$, by the case of $SO(2)$ invariance, $\mathbb{C}_2^{SO(2)}$ is a dimension 2 linear space generated by $X^2+Y^2$ and $XY-YX$.

            Under an arbitrary $SU(2)$ action $u = \begin{pmatrix}
                a&b\\-\bar{b}&\bar{a}
            \end{pmatrix}$ with $|a|^2+|bI^2=1$, $u\cdot (X,Y) = (aX+bY,-\bar{b}X+\bar{a}Y)$. Then for $f_1(X,Y) = X^2+Y^2$,
            \begin{align*}
                &f_1\circ u(X,Y)\\
                =&f_1(aX+bY,-\bar{b}X+\bar{a}Y)\\
                =&(aX+bY)^2+(-\bar{b}X+\bar{a}Y)^2\\
                =&(a^2+\bar{b}^2)X^2+(b^2+\bar{a}^2)Y^2+(ab-\overline{ab})XY+(ab-\overline{ab})YX\\
                \neq&X^2+Y^2
            \end{align*}
            Note that when both $a$ and $b$ are real numbers, the last inequality becomes an equality. This aligns with our case of $SO(2)$ invariance.

            Similarly, consider $f_2(X,Y) = XY-YX$,
            \begin{align*}
                &f_2\circ u(X,Y)\\
                =&f_2(aX+bY,-\bar{b}X+\bar{a}Y)\\
                =&(aX+bY)(-\bar{b}X+\bar{a}Y) - (-\bar{b}X+\bar{a}Y)(aX+bY)\\
                =&(-a\bar{b}+a\bar{b})X^2+(b\bar{a}-b\bar{a})YX+(a\bar{a}+b\bar{b})XY-(a\bar{a}+b\bar{b})YX\\
                =&XY-YX
            \end{align*}
            Therefore $\mathbb{C}_2^{SU(2)}$ is the 1 dimensional linear subspace generated by $XY-YX$.

            Since the only non trivial element in $S_2$ is $\eta=(1,2)$ and
            \begin{equation*}
                \eta \curvearrowright (XY-YX) = YX-XY = -(XY-YX)
            \end{equation*}
            So $\mathbb{C}_2^{SU(2)} = span\{\eta\curvearrowright(XY-YX) | \eta\in S_{2}\}$

            Now assume when $n=k$ the theorem holds, that is
            \begin{equation*}
                \mathbb{C}_{2k}^{SU(2)} = span\{ \eta \curvearrowright(XY-YX)^k | \eta\in S_{2k}\}
            \end{equation*}
            It remains to show
            \begin{equation*}
                \mathbb{C}_{2k+2}^{SU(2)} = span\{ \eta \curvearrowright(XY-YX)^{k+1} | \eta\in S_{2k+2}\}
            \end{equation*}

            Since $SD(2)\subset SU(2)$,
            \begin{equation*}
                \mathbb{C}_{2k+2}^{SU(2)}\subset\mathbb{C}_{2k+2}^{SD(2)}
            \end{equation*}
            Therefore for any non commutative polynomial $F(X,Y)\in \mathbb{C}_{2k+2}^{SU(2)}$,
            \begin{equation*}
                F = \sum c_jM_j
            \end{equation*}
            where $M_j\in\mathbb{C}_{2k+2}^{SD(2)}$ are words in $X$ and $Y$.

            Then
            \begin{equation*}
                F = I(F) = I(\sum c_jM_j) = \sum c_jI(M_j)
            \end{equation*}

            By Lemma \ref{I(M-N)}, for any two $I(M_j)$, there difference is some element in $span\{ \eta \curvearrowright(XY-YX)^{2k+2} | \eta\in S_{2k+2}\}$. Therefore we only need to show $I(M)\in span\{ \eta \curvearrowright(XY-YX)^{2k+2} | \eta\in S_{2k+2}\}$ for one choice of $M\in \mathbb{C}_{2k+2}^{SD(2)}$.

            Fix a $M\in \mathbb{C}_{2k+2}^{SD(2)}$. Consider $\binom{2k+2}{k+1}I(M)$. Let $\Big\{M_j\Big\}_{j=1}^{\binom{2k+2}{k+1}}$. By Lemma \ref{I(M-N)},
            \begin{align*}
                \binom{2k+2}{k+1}I(M) &= F+I(\sum_{j=1}^{\binom{2k+2}{k+1}}M_j)\\
                &= F+\int_{SU(2)}\sum_{j=1}^{\binom{2k+2}{k+1}}M_j\circ u\;dh(u)
            \end{align*}
            Here $F\in span\{ \eta \curvearrowright(XY-YX)^{2k+2} | \eta\in S_{2k+2}\}$.
            Let $v = \begin{pmatrix}
                \frac{1}{\sqrt{2}}&\frac{i}{\sqrt{2}}\\ \frac{i}{\sqrt{2}}&\frac{1}{\sqrt{2}}
            \end{pmatrix} \in SU(2)$. Since $dh$ is Haar measure,
            \begin{equation*}
                \int_{SU(2)} F\circ u\;dh(u) = \int_{SU(2)}(F\circ v)\circ u\;dh(u)
            \end{equation*}
            Therefore
            \begin{align*}
                \binom{2k+2}{k+1}I(M) &=F+\int_{SU(2)}\sum_{j=1}^{\binom{2k+2}{k+1}}M_j\circ u\;dh(u)\\
                &= F+\int_{SU(2)}\sum_{j=1}^{\binom{2k+2}{k+1}}(M_j\circ v)\circ u\;dh(u)\\
                &= F+\int_{SU(2)}\sum_{j=1}^{\binom{2k+2}{k+1}}\Big(\sum_{k=1}^{\binom{2k+2}{k+1}}\Delta(M_j,M_k)M_k\Big)\circ u\;dh(u) + \int_{SU(2)}G\circ u\;dh(u)\\
                &= F+\int_{SU(2)}\sum_{k=1}^{\binom{2k+2}{k+1}}\Big(\sum_{j=1}^{\binom{2k+2}{k+1}}\Delta(M_j,M_k)\Big)M_k\circ u\;dh(u) + \int_{SU(2)}G\circ u\;dh(u)
            \end{align*}
            Here $\Delta(M_j,M_k)$ is defined as $\frac{(-1)^{\delta(M_j,M_k)}}{2^{k+1}}$, and $\delta(M_j,M_k)$ is the number of places at which $M_j$ is different from $M_k$ given that $M_j$'s are words in $X$ and $Y$. The function $G$ in the third term is a non commutative polynomial in $\mathbb{C}_{2k+2}\ominus\mathbb{C}_{2k+2}^{SD(2)}$. By \ref{non homo int}, $\int_{SU(2)}G\circ u\;dh(u)=0$

            Thus we only need to check the second term. Because
            \begin{equation*}
                \sum_{j=1}^{\binom{2k+2}{k+1}}\Delta(M_j,M_k)=\frac{1}{2^{k+1}}\Big(\sum_{n=0}^{k+1}\binom{k+1}{n}(-1)^N\Big)=0
            \end{equation*}
            the second term is also zero.

            Thus
            \begin{equation*}
                \binom{2k+2}{k+1}I(M) =F \in span\{ \eta \curvearrowright(XY-YX)^{2k+2} | \eta\in S_{2k+2}\}
            \end{equation*}

            This completes the induction.
        \end{proof}

        Notice that $\frac{1}{n+1}\binom{2n}{n}$ is the n-th Catalan number. We show the second part of Theorem \ref{conjSU2} by the following inductive construction.

        \begin{theorem}\label{part2}
            Define $P_n$ in the following way,
            \begin{equation}
                P_n = \{ XAYB-YAXB\Big|A\in P_k, B\in P_{n-k-1}  \}
            \end{equation}
            Then $P_n$ is a basis for $\mathbb{C}^{SU(2)}_{2n}$
        \end{theorem}

        For example,
        \begin{align*}
            P_0 &= \{1\}\\P_1&=\{XY-YX\}\\P_2 &= \{XY(XY-YX)-YX(XY-YX)\;,\;X(XY-YX)Y-Y(XY-YX)X\}
        \end{align*}

        Let $|P_n| = a_n$, then $a_n$ satisfy
        \begin{equation} \label{cat}
            a_n = \sum_{k=0}^{n-1}a_ka_{n-k-1}
        \end{equation}
        with $a_0=1$. By \cite{Cat} Chapter 1.2, \ref{cat} is the recurrence relation for Catalan numbers, so $a_n$ is exactly the n-th Catalan number $\frac{1}{n+1}\binom{2n}{n}$.

        \begin{proof}
            We use induction again to prove the theorem. It is true that
            \begin{equation}
                P_1 = \{ XAYB-YAXB\Big|A\in P_0, B\in P_{0}  \}=\{XY-YX\}
            \end{equation}
            Suppose
            \begin{equation}
                P_n = \{ XAYB-YAXB\Big|A\in P_k, B\in P_{n-k-1}  \}
            \end{equation}
            is true for all $n = 1,\ldots,k$. Consider the case of $n=k+1$.
            By the argument of \ref{I(M-N)}, if $M=PYXQ$ and $N=PXYQ$ with the degree of $P$ being $i-1$, then
            \begin{equation*}
                I(M-N)=\eta \curvearrowright(XY-YX)I(PQ)
            \end{equation*}
            Here $\eta$ is the permutation that swaps 1 with $i$, swaps 2 with $i+1$, and keeps every other nodes unchanged. We write $\eta = (1\;i)(2\;[i+1])$.

            Since $I(PQ)\in \mathbb{C}^{SU(2)}_{2k}$, by the induction hypothesis,
            \begin{equation*}
                I(PQ) = \sum c_j(XA_jYB_j-YA_jXB_j)
            \end{equation*}
            where each pair $(A_j,B_j)$ is in $P_l\times P_{k-l-1}$ for some $l$.
            Consider
            \begin{equation*}
                \eta \curvearrowright\Big(XY-YX)(XA_jYB_j-YA_jXB_j)\Big)\;.
            \end{equation*}
            Assume the degree of $A_j$ is $2l$ and the degree of $B_j$ is $2k-2l-2$. 
            
            If $i=1$, then
            \begin{equation*}
                \eta \curvearrowright\Big(XY-YX)(XA_jYB_j-YA_jXB_j)\Big) = (XY-YX)(XA_jYB_j-YA_jXB_j)
            \end{equation*}
            This is in the form of $XAYB-YAXB$ with
            \begin{align*}
                A&=1 \in P
                _0\\
                B &= XA_jYB_j-YA_jXB_j\in \mathbb{C}^{SU(2)}_{2k} \in span(P_k)
            \end{align*}

            If $2\leq i \leq 2l+1$, then
            \begin{align*}
                &\eta \curvearrowright\Big(XY-YX)(XA_jYB_j-YA_jXB_j)\Big)\\
                =& X\Big(\eta'\curvearrowright (XY-YX)A_j \Big)YB_j-Y\Big(\eta'\curvearrowright (XY-YX)A_j \Big)XB_j\;,
            \end{align*}
            where $\eta' = (1\;[i-1])(2\;i)\in S_{2l+2}$.

            So
            \begin{equation*}
                \eta \curvearrowright\Big(XY-YX)(XA_jYB_j-YA_jXB_j)\Big)
            \end{equation*}
            is also in form of $XAYB-YAXB$, with
            \begin{align*}
                A&=\eta'\curvearrowright (XY-YX)A_j \Big) \in span(P_{l+1})\\
                B&=B_j\in span(P_{k-l-1})\;.
            \end{align*}

            Similarly, if $i \geq 2l+2$, then
            \begin{align*}
                &\eta \curvearrowright\Big(XY-YX)(XA_jYB_j-YA_jXB_j)\Big)\\
                =& XA_jY\Big(\eta'\curvearrowright (XY-YX)B_j \Big)-YA_jX\Big(\eta'\curvearrowright (XY-YX)B_j \Big)\;,
            \end{align*}
            where $\eta' = (1\;[i-2l-2])(2\;[i-2l-1])\in S_{2k-2l}$. So $\eta \curvearrowright\Big(XY-YX)(XA_jYB_j-YA_jXB_j)\Big)$ is also in the form of $XAYB-YAXB$, with
            \begin{align*}
                A&=A_j \in span(P_{l})\\
                B&=\eta'\curvearrowright (XY-YX)B_j  \in span(P_{k-l})\;.
            \end{align*}

            Thus
            \begin{align*}
                span(P_n)&\subseteq span\{ \eta \curvearrowright(XY-YX)^{2n} | \eta\in S_{2n}\} \\&= \mathbb{C}_{2n}^{SU(2)}<X,Y>\\&\subseteq span(P_n)
            \end{align*}

            Hence $\mathbb{C}_{2n}^{SU(2)}<X,Y>\subseteq span(P_n)$.

            By the following lemma, elements in $P_n$ are linearly independent, therefore it can be concluded that the dimension of $\mathbb{C}^{SU(2)}_{2n}<X,Y>$ is the n-th Catalan number $a_n= \frac{1}{n+1}\binom{2n}{n}$
        \end{proof}

        \begin{lemma}
            Elements in $P_n$ are linearly independent.
        \end{lemma}

        \begin{proof}
            In order to show elements in $P_n$ are linearly independent, we only need to show the coefficient matrix determined by polynomials in $P_n$ has full rank.

            Instead of checking all coefficients in each polynomial in $P_n$, we show it is sufficient to consider a smaller number of coefficients, so only a square matrix need to be checked. This smaller matrix will be lower triangular and all entries on diagonal will be 1.

            We first give the polynomials in $P_n$ an ordering. Recall that all element in $P_n$ are of the form
            \begin{equation*}
                p(X,Y) = XAYB-YAXB
            \end{equation*}
            with $A\in P_l$, $B\in P_{n-l-1}$.

            Define an order on $P_n$ inductively by the following procedure. If $p_j(X,Y)=XA_jYB_j-YA_jXB_j$ with $A\in P_{l_j}$, $B\in P_{n-l_j-1}$, then $p_j \prec p_k$ when (1) $l_j<l_k$, or (2) $l_j=l_k$ and $A_j \prec A_k$, or (2) $l_j=l_k$ and $A_j = A_k$ and $B_j\prec B_k$. This is a total order, so for all elements in $P_n$, there is an indexing $p_{n,j}$ such that if $j<k$, then $p_{n,j} \prec p_{n,k}$.

            For example,
            \begin{align*}
                p_{0,1} &= 1\\ \\
                p_{1,1} &= Xp_{0,1}Yp_{0,1} - Yp_{0,1}Xp_{0,1} = XY-YX\\
                p_{2,1} &= Xp_{0,1}Yp_{1,1} - Yp_{0,1}Xp_{1,1}\\
                &=XY(XY-YX)-YX(XY-YX) = (XY-YX)^2\\
                p_{2,2} &= Xp_{1,1}Yp_{0,1} - Yp_{1,1}Xp_{0,1}\\
                &= X(XY-YX)Y-Y(XY-YX)X\\ \\
                p_{3,1} &= Xp_{0,1}Yp_{2,1} - Yp_{0,1}Xp_{2,1}\\
                &= XY\Big(XY(XY-YX)-YX(XY-YX)\Big)\\&-YX\Big(XY(XY-YX)-YX(XY-YX)\Big) = (XY-YX)^3\\
                p_{3,2} &= Xp_{0,1}Yp_{2,2} - Yp_{0,1}Xp_{2,2}\\
                &=XY\Big(X(XY-YX)Y-Y(XY-YX)X\Big)\\&-YX\Big(X(XY-YX)Y-Y(XY-YX)X\Big)\\
                p_{3,3} &= Xp_{1,1}Yp_{1,1} - Yp_{1,1}Xp_{1,1}\\
                &=X(XY-YX)Y(XY-YX)-Y(XY-YX)X(XY-YX)\\
                p_{3,4} &= Xp_{2,1}Yp_{0,1} - Yp_{2,1}Xp_{0,1}\\
                &=X\Big(XY(XY-YX)-YX(XY-YX)\Big)Y\\
                &-Y\Big(XY(XY-YX)-YX(XY-YX)\Big)X\\
                p_{3,5} &= Xp_{2,2}Yp_{0,1} - Yp_{2,2}Xp_{0,1}\\
                &=X\Big(X(XY-YX)Y-Y(XY-YX)X\Big)Y\\
                &-Y\Big(X(XY-YX)Y-Y(XY-YX)X\Big)X\\ \\
                p_{4,1} &= Xp_{0,1}Yp_{3,1} - Yp_{0,1}Xp_{3,1}\\
                &= (XY-YX)^4
            \end{align*}

            Now in a similar manner, define a set of words $Q_n$ by
            \begin{equation*}
                Q_n = \{XAYB| A\in Q_l, B\in Q_{n-l-1}\}
            \end{equation*}
            Again, an ordering can be defined on $Q_n$. If $q_j(X,Y)=XA_jYB_j$ with $A\in Q_{l_j}$, $B\in Q_{n-l_j-1}$, then $q_j \prec q_k$ when (1) $l_j<l_k$, or (2) $l_j=l_k$ and $A_j \prec A_k$, or (2) $l_j=l_k$ and $A_j = A_k$ and $B_j\prec B_k$. This is a total order, so for all elements in $Q_n$, there is an indexing $q_{n,j}$ such that if $j<k$, then $q_{n,j} \prec q_{n,k}$.

            For example,
            \begin{align*}
                q_{0,1} &=1\\
                q_{1,1} &= Xq_{0,1}Yq_{0,1} = XY\\
                q_{2,1} &= Xq_{0,1}Yq_{1,1}= XYXY\\
                q_{2,2} &= Xq_{1,1}Yq_{0,1}= XXYY\\
                q_{3,1} &= Xq_{0,1}Yq_{2,1}= XYXYXY\\
                q_{3,2} &= Xq_{0,1}Yq_{2,2}= XYXXYY\\
                q_{3,3} &= Xq_{1,1}Yq_{1,1}= XXYYXY\\
                q_{3,4} &= Xq_{2,1}Yq_{0,1}= XXYXYY\\
                q_{3,5} &= Xq_{2,2}Yq_{0,1}= XXXYYY\\
            \end{align*}

            Now a polynomial in $P_n$ can be written as
            \begin{equation*}
                p_{n,j}(X,Y) = \sum_{k}c_{n,j,k}q_{n,k}(X,Y) + W
            \end{equation*}
            where $W$ is a polynomial made of words that are not in $Q_n$. We want to check the rank of matrix $C_n=\{c_{n}\}_{j,k}$.

            For example,
            \begin{align*}
                C_0 &= \begin{pmatrix}
                    1
                \end{pmatrix}\\
                C_1 &= \begin{pmatrix}
                    1
                \end{pmatrix}\\
                C_2 &= \begin{pmatrix}
                    1&0\\-1&1
                \end{pmatrix}\\
                C_3 &= \begin{pmatrix}
                    1&0&0&0&0\\
                    -1&1&0&0&0\\
                    -1&0&1&0&0\\
                    1&-1&-1&1&0\\
                    -1&0&0&-1&0
                \end{pmatrix}
            \end{align*}

            Along the diagonal of $C_n$, there are block matrices $C_0\otimes C_{n-1}$, $C_1\otimes C_{n-2}$, ..., $C_{n-1}\otimes C_0$, so the diagonal of $C_n$ are all one's. For the entries above diagonal, by the ordering of $p_{n,j}$ and $q_{n,j}$, they are all zero's. So $det(C_n)=1\neq 0$, hence $P_n$ the set of nc polynomials are linearly independent.
        \end{proof}

        Theorem \ref{part1} and \ref{part2} combined yield \ref{conjSU2}.

	\subsection{$SU(n)$}\label{sec:SUn}

We now let $SU(n)$ act on $X^{(n)}=(X_1,\dots,X_n)$ as in Section~1.2, and
write $f\circ u$ for $f(u\cdot X^{(n)})$. Permutations act on positions with
the convention of Lemma~\ref{etaI}: for a word $A_1A_2\cdots A_d$ and
$\eta\in S_d$,
\[
  \eta\curvearrowright(A_1A_2\cdots A_d)=A_{\eta^{-1}(1)}A_{\eta^{-1}(2)}\cdots A_{\eta^{-1}(d)},
\]
so the letter in position $p$ moves to position $\eta(p)$.

\begin{definition}\label{def:Dn-Ok}
Let
\[
  D_n:=\sum_{\eta\in S_n}\operatorname{sgn}(\eta)\,\eta\curvearrowright(X_1X_2\cdots X_n),
\]
and for $k\geq0$ let
\[
  \mathcal O_k:=\operatorname{span}\{\mu\curvearrowright D_n^{\,k}\mid \mu\in S_{kn}\}
  \subseteq\mathbb C_{kn}\langle X^{(n)}\rangle .
\]
\end{definition}

Thus $D_2=XY-YX$, and for $n=2$, $\mathcal O_k$ is the space appearing in
Theorem~\ref{part1}. Let $dh$ be the normalized Haar measure on $SU(n)$
and, as in Definition~\ref{def:avgI},
\[
  I(f):=\int_{SU(n)}f\circ u\,dh(u).
\]

\begin{lemma}\label{lem:SUn-I}
Let $f\in\mathbb C\langle X^{(n)}\rangle$ be homogeneous of degree $d$.
\begin{enumerate}
  \item[(a)] $I(f)\in\mathbb C^{SU(n)}\langle X^{(n)}\rangle$, and $I(f)=f$ if $f$ is
        $SU(n)$-invariant.
  \item[(b)] For $\eta\in S_d$ and $u\in SU(n)$,
        $(\eta\curvearrowright f)\circ u=\eta\curvearrowright(f\circ u)$; hence
        $I(\eta\curvearrowright f)=\eta\curvearrowright I(f)$.
  \item[(c)] If $f$ is a word not in $\mathbb C^{SD(n)}\langle X^{(n)}\rangle$, then $I(f)=0$.
\end{enumerate}
\end{lemma}

\begin{proof}
(a) is the computation following Definition~\ref{def:avgI}, with $SU(2)$
replaced by $SU(n)$. For (b), the proof of Lemma~\ref{etaI} expands
$f\circ u$ by substituting each letter separately, and this works for any
$u\in SU(n)$.

For (c), argue as in Lemma~\ref{Ikernel}. Let $d_j$ be the degree of $X_j$
in $f$, and choose $a,b$ with $d_a\neq d_b$. Take
$v\in SD(n)$ with $v_{aa}=e^{it}$, $v_{bb}=e^{-it}$ and all other diagonal
entries equal to $1$. Then $f\circ v=e^{it(d_a-d_b)}f$, and left invariance of
$dh$ gives
\[
  I(f)=I(f\circ v)=e^{it(d_a-d_b)}I(f).
\]
Choosing $t$ with $e^{it(d_a-d_b)}\neq1$ gives $I(f)=0$.
\end{proof}

By (a) and (c), $\mathbb C^{SU(n)}_d\langle X^{(n)}\rangle$ is the image under
$I$ of the span of the words in $\mathbb C^{SD(n)}_d\langle X^{(n)}\rangle$. By
Theorem~\ref{balancedegree}, there are no such words unless $n\mid d$.

\begin{corollary}\label{cor:SUn-degree}
$\mathbb C^{SU(n)}_d\langle X^{(n)}\rangle=0$ unless $n$ divides $d$.
\end{corollary}

\begin{lemma}\label{lem:Dn-inv}
$D_n\circ u=D_n$ for every $u\in SU(n)$. Consequently
$\mathcal O_k\subseteq\mathbb C^{SU(n)}_{kn}\langle X^{(n)}\rangle$ for all $k\geq0$.
\end{lemma}

\begin{proof}
Substituting $X_a\mapsto\sum_bu_{ab}X_b$ gives
\[
  D_n\circ u=\sum_{j_1,\dots,j_n=1}^{n}
  \Big(\sum_{\eta\in S_n}\operatorname{sgn}(\eta)\prod_{t=1}^nu_{\eta(t)\,j_t}\Big)
  X_{j_1}\cdots X_{j_n}.
\]
The coefficient of $X_{j_1}\cdots X_{j_n}$ is the determinant of the matrix
whose $t$-th column is the $j_t$-th column of $u$. It is $0$ if two of the $j_t$
coincide. If $(j_1,\dots,j_n)$ is a permutation $\rho$, it equals
$\operatorname{sgn}(\rho)\det u=\operatorname{sgn}(\rho)$. Hence $D_n\circ u=D_n$.

Since $f\mapsto f\circ u$ is multiplicative, $D_n^{\,k}$ is invariant. By
Lemma~\ref{lem:SUn-I}(a),(b),
$I(\mu\curvearrowright D_n^{\,k})=\mu\curvearrowright D_n^{\,k}$, so every
$\mu\curvearrowright D_n^{\,k}$ is invariant as well.
\end{proof}

\begin{theorem}\label{thm:SUn-degree-n}
Up to scalar multiplication, $D_n$ is the only $SU(n)$-invariant homogeneous
nc polynomial of degree $n$:
\[
  \mathbb C^{SU(n)}_n\langle X^{(n)}\rangle=\operatorname{span}\{D_n\}.
\]
\end{theorem}

\begin{proof}
By Lemma~\ref{lem:Dn-inv}, $D_n$ is invariant. Let
$f\in\mathbb C^{SU(n)}_n\langle X^{(n)}\rangle$ be nonzero. By
Lemma~\ref{lem:SUn-I}(a),(c), $f$ is a combination of words in
$\mathbb C^{SD(n)}_n\langle X^{(n)}\rangle$, that is, words in which each of
$X_1,\dots,X_n$ appears exactly once.

\emph{Normalization.} Pick a word $X_{\sigma_0(1)}\cdots X_{\sigma_0(n)}$ with
nonzero coefficient $c$ in $f$. With the convention above,
$\sigma_0\curvearrowright(X_{\sigma_0(1)}\cdots X_{\sigma_0(n)})=X_1X_2\cdots X_n$.
Put
\[
  \tilde f:=c^{-1}\,\sigma_0\curvearrowright f .
\]
By Lemma~\ref{lem:SUn-I}(a),(b), $\tilde f$ is again $SU(n)$-invariant, and the
coefficient of $X_1X_2\cdots X_n$ in $\tilde f$ is $1$. For $\sigma\in S_n$ write
$X_\sigma:=X_{\sigma(1)}X_{\sigma(2)}\cdots X_{\sigma(n)}$, so that
$X_{\mathrm{id}}=X_1X_2\cdots X_n$, and
\[
  \tilde f=\sum_{\sigma\in S_n}c_\sigma X_\sigma,
  \qquad c_{\mathrm{id}}=1 .
\]
In Steps 1--3 we show that $c_\sigma=\operatorname{sgn}(\sigma)$ for every
$\sigma$, that is, $\tilde f=D_n$.

Since $\eta\curvearrowright D_n=\operatorname{sgn}(\eta)D_n$ for every $\eta\in S_n$,
it then follows that
\[
  f=c\,\sigma_0^{-1}\curvearrowright D_n=c\operatorname{sgn}(\sigma_0)\,D_n .
\]

\emph{Step 1: a subgroup $SU(2)\oplus I$.} For $1\leq a\leq n-1$ let
\[
  H_a:=\{\,I_{a-1}\oplus v\oplus I_{n-a-1}\mid v\in SU(2)\,\}\subseteq SU(n).
\]
An element of $H_a$ fixes $X_c$ for $c\neq a,a+1$, and acts on $(X_a,X_{a+1})$
exactly as $v\in SU(2)$ acts on $(X,Y)$. Since $\tilde f$ is $SU(n)$-invariant, it
is in particular $H_a$-invariant.

\emph{Step 2: $c_{(a,\,a+1)\sigma}=-c_\sigma$.} Let
\[
  w_a:=I_{a-1}\oplus\begin{pmatrix}0&1\\-1&0\end{pmatrix}\oplus I_{n-a-1}\in H_a\subseteq SU(n).
\]
Then $w_a$ acts by
\[
  X_a\mapsto X_{a+1},\qquad X_{a+1}\mapsto -X_a,\qquad X_c\mapsto X_c\ (c\neq a,a+1).
\]
Every word $X_\sigma$ contains each of $X_a$ and $X_{a+1}$ exactly once. Hence
$X_\sigma\circ w_a=-X_{(a,\,a+1)\sigma}$, where $X_{(a,\,a+1)\sigma}$ is the word
obtained from $X_\sigma$ by exchanging the letters $X_a$ and $X_{a+1}$. So
\[
  \tilde f\circ w_a=-\sum_{\sigma\in S_n}c_\sigma X_{(a,\,a+1)\sigma} .
\]
That is, $\tilde f\circ w_a$ is $-\tilde f$ with $X_a$ and $X_{a+1}$
interchanged. Since $\tilde f$ is invariant, $\tilde f\circ w_a=\tilde f$, so
$\tilde f$ is antisymmetric under the exchange of $X_a$ and $X_{a+1}$. Comparing
the coefficients of $X_{(a,\,a+1)\sigma}$ gives
\[
  c_{(a,\,a+1)\sigma}=-c_\sigma .
\]

\emph{Step 3: from $X_{\mathrm{id}}$ to $X_\sigma$.} The adjacent transpositions
$(a,\,a+1)$, $1\leq a\leq n-1$, generate $S_n$. So every $\sigma\in S_n$ can be
written as
\[
  \sigma=(a_m,\,a_m+1)\cdots(a_2,\,a_2+1)(a_1,\,a_1+1),
\]
and $X_\sigma$ is reached from $X_{\mathrm{id}}$ by $m$ successive swaps of two
adjacent letters. Applying Step~2 at each swap gives
\[
  c_\sigma=(-1)^m c_{\mathrm{id}}=(-1)^m=\operatorname{sgn}(\sigma),
\]
since $\operatorname{sgn}$ is a homomorphism with $\operatorname{sgn}\big((a,\,a+1)\big)=-1$.
In particular, the value obtained does not depend on the chosen path from
$\mathrm{id}$ to $\sigma$.
\end{proof}

We conjecture that the pattern of Theorem~\ref{conjSU2} persists for all $n$.

\begin{conjecture}\label{conj:SUn}
For all $n\geq2$ and $k\geq0$,
\[
  \mathbb C^{SU(n)}_{kn}\langle X^{(n)}\rangle
  =\operatorname{span}\{\mu\curvearrowright D_n^{\,k}\mid\mu\in S_{kn}\}=\mathcal O_k .
\]
\end{conjecture}

By Lemma~\ref{lem:Dn-inv}, only the inclusion $\subseteq$ is at stake. The
conjecture holds for $n=2$ (Theorem~\ref{conjSU2}) and for $k=1$
(Theorem~\ref{thm:SUn-degree-n}).  The point of what
follows is to approach it by the Haar-integral method of Section~3.5 and to find a explicit basis that suggest a realization formula.

    \section{Summary}\label{sec4}
        We studied noncommutative polynomials invariant under infinite (Lie) group actions on two and several noncommuting variables. After showing (Theorem 2.1) that the algebraic classification of invariant nc polynomials automatically controls invariant free holomorphic functions nonvanishing at the origin, we developed the algebraic theory case by case. 
        
        For the torus $SD(n)$ (\S3.1) and $SO(2)$ (\S3.2), we gave an explicit description in terms of the degree of each variable and the eigenfunctions $X\pm iY$ respectively, together with a closed dimension formula. Using the matrix decomposition $U=ADB$ with $A,B\in SO(n)$ and $D\in SD(n)$ (\S3.3), we showed $\mathbb{C}^{SU(n)}=\mathbb{C}^{SO(n)}\cap\mathbb{C}^{SD(n)}$, reducing the study of $SU(n)$-invariants to these two simpler group actions. 
        
        For $SU(2)$ (\S\S3.4--3.5), we gave a complete classification: the degree-$2n$ invariant polynomials are exactly the $S_{2n}$-orbit of $(XY-YX)^n$ (Theorem~\ref{conjSU2}), with an explicit recursive basis $P_n$ (Theorem~\ref{part2}) whose cardinality is the $n$-th Catalan number, and odd-degree invariants vanish. 

        Several directions remain open. The most immediate is our conjecture (\S3.6) that for general $SU(n)$, the degree-$kn$ invariants are spanned by the $S_{kn}$-orbit of the $k$-th power of the unique degree-$n$ invariant $\sum_{\eta\in S_n}\mathrm{sgn}(\eta)\,\eta\curvearrowright(X_1\cdots X_n)$ as well as the dimension counting. This problem, unlike the $SU(2)$ case, remain to be established by the Haar-integral method for $n>2$ for both the spanning statement and an explicit linearly independent basis . 
        
        A second direction is to extend the present results from the polynomial ring to the skew field of rational noncommutative functions and check if it is finitely generated in any sense, in the spirit of Agler and Young's finite-basis result for the finite-group case \cite{2vsym} and its subsequent generalizations \cite{nvsym,nvsymc,ncsymext}. Theorem 2.1 already supplies part of the analytic bridge such an extension would need. 
        
        Finally, it would be of interest to carry out the Haar-integral averaging combined with a matrix decomposition into simpler subgroups methodology for other classical Lie groups, such as $Sp(n)$ and $O(n)$, acting on noncommuting variables.

    \section*{AI usage declaration}
    AI(Claude Sonnet 5) is used for section \ref{sec1.3}: Link with known results and section \ref{sec4}: Summary. The author also used AI for language support on the entire paper. None of the theorems is proved with assistance from AI.
\printbibliography
\end{document}